\documentclass[11pt,reqno]{amsart}
\usepackage{pbruillard_commands}
\usepackage{gustafson_commands}

\usepackage[all,knot]{xy}

\newcommand{\pjbcomment}[1]{\color{red}#1 -pjb\color{black}}
\newcommand{\paren}[1]{\left(#1\right)}
\renewcommand{\diag}{\mathrm{diag}}
\newcommand{\intcat}{\mathrm{int}}
\newcommand{\one}{\mathbf{1}}
\newcommand{\ot}{\otimes}
\newcommand{\B}{\mcB}

\begin{document}
\title[Dimension and metaplectic categories]{Dimension as a quantum statistic and the classification of metaplectic categories}
\date{\today}
\author{Paul Bruillard}
\email{Paul.Bruillard@pnnl.gov}
\author{Paul Gustafson}
\email{pgustafs@math.tamu.edu}
\author{Julia Yael Plavnik}
\email{julia@math.tamu.edu}
\author{Eric C. Rowell}
\email{rowell@math.tamu.edu}

\thanks{\textit{PNNL Information Release:} PNNL-SA-120943}
\thanks{
The research in this paper was, in part, conducted under the Laboratory Directed
Research and Development Program at PNNL, a multi-program national laboratory
operated by Battelle for the U.S. Department of Energy.}


\commby{}

\begin{abstract} We discuss several useful interpretations of the categorical dimension of objects in a braided fusion category, as well as some conjectures demonstrating the value of quantum dimension as a quantum statistic for detecting certain behaviors of anyons in topological phases of matter.  From this discussion we find that objects in braided fusion categories with integral \textit{squared} dimension have distinctive properties.  A large and interesting class of non-integral modular categories such that every simple object has integral squared-dimensions are the metaplectic categories that have the same fusion rules as $SO(N)_2$ for some $N$. We describe and complete their classification and enumeration, by recognizing them as $\ZZ_2$-gaugings of cyclic modular categories (i.e. metric groups).  We prove that any modular category of dimension $2^km$ with $m$ square-free and $k\leq 4$, satisfying some additional assumptions, is a metaplectic category.  This illustrates anew that dimension can, in some circumstances, determine a surprising amount of the category's structure.  
 \end{abstract}

\maketitle

\section{Introduction}
  \label{Section: Introduction}
  
  The goal of this article is twofold: 1) to survey various notions of dimension for fusion categories and some properties determined (at least conjecturally) by dimension and 2) to classify of metaplectic categories of dimension $16k$ and apply it to the classification of categories with global dimension $16m$ with $m$ square-free.

  Dimensions of simple objects in fusion categories are one of the most ubiquitous invariants we encounter. Algebraically, a  dimension function on a fusion category $\mcC$ is a generalization of a linear character for a finite group: it is an assignment of a complex number to each object $X\in\mcC$ that obeys the fusion rules.  As the character table of a finite group $G$ contains a significant amount of information about $G$ itself (e.g. whether $G$ is abelian, simple, or perfect), it is natural to ask: What information is contained in the dimensions of simple objects in a category?  That is, how much of the structure, and which properties are determined by dimensions? The two dimension functions that one is most often interested in are the categorical and Frobenius-Perron (FP) dimensions. Fortunately, in the physically-relevant unitary setting the FP and categorical dimension coincide. For our purposes it is the FP dimension that is most useful, so we will focus on this particular function.  In what follows we will simply refer to the FP dimension as the dimension.
  
  Non-degenerate ribbon fusion (i.e. modular) categories model certain topological phases of matter \cite{Nayaketal}, where the simple objects label anyons, and the isomorphism classes correspond to the anyon types. Here the word anyon is meant to convey that they are generalizations of bosons and fermions, which are (quasi)-particles with Bose/Fermi exchange statistics. In this interpretation the quantum dimensions of simple objects correspond asymptotically to the dimensions of state spaces of $n$ identical anyons in a disk with boundary labeled by the trivial object (vacuum anyon).  One may also ask: What properties of an anyon are determined by their (label's) dimension?  As we view anyon systems through the lens of category theory, this leads to mathematical questions and conjectures.  There are three important and seemingly unrelated properties of an anyon that are (at least conjecturally) controlled by the dimension: abelianness, localizability and universality.  This motivates our perspective that dimension is the central quantum statistic for anyonic systems.  We have the following (see Section \ref{Section: Properties}):
  
  \begin{conjnn} An anyon $x$ is braiding universal if and only if $\dim(x)^2\in\mbbZ$, if and only if $x$ is localizable.
  \end{conjnn}

 A coarser invariant of a fusion category is the global (FP) dimension: the sum of the squares of the dimension of the simple objects.  Still, some properties of a modular category are determined by the global dimension.  It is well-known (see \cite[Cor. 3.14]{BNRW1}) that there are finitely many fusion categories with a given fixed global dimension.  In some cases one can go a bit further to give a complete classification of modular categories of a given global dimension in terms of other well-studied categories.  We illustrate this principle by classifying certain modular categories of dimension $16m$ with $m$ square-free in terms of metaplectic modular categories.

A critical ingredient of this classification is the classification of metaplectic modular categories, which we finish in this paper, section \ref{Subsection: Metaplectic}.
Combining with results of \cite{ACRW1,BPR} we have the following
\begin{thm} Let $\mcC$ be a metaplectic modular category of dimension $4N$.  Let $N=2^ap_1^{a_1}\cdots p_s^{a_s}$ be the prime factorization (where $a\geq 0$ and $a_i>0$). Then 

\begin{enumerate}
\item[(a)] $\mcC$ is obtained from a pointed cyclic modular category $\mcC(\mbbZ_N,q)$ by a $\mbbZ_2$-gauging.
\item[(b)] If $a\leq 1$ then there are exactly $2^{s+1+a}$ inequivalent metaplectic modular categories of dimension $4N$, and if $a>1$ there are $3(2^{s+2})$. \end{enumerate}
\end{thm}

As we mentioned above, for strictly weakly integral modular categories the structure can sometimes be determined by the dimension.  Of course this is far from true for integral modular categories. There are many inequivalent modular categories of dimension $2^{2n}$: simply take the twisted double of a finite group of order $2^n$.  On the other hand, the dimension of a non-integral category $\mcC$ can be factored over the Dedekind domain $\mbbZ[\zeta_k]$ where $k$ is the order of the $T$-matrix (see \cite{BNRW1}) and it makes sense to look at such prime factorizations.  For example:

\begin{question} How much of the structure of a non-integral modular category $\mcC$ can be determined from the primes dividing the ideal $\langle\dim\mcC\rangle$ (in an appropriate $\mbbZ[\zeta_k]$)?
\end{question}

As metaplectic modular categories are ubiquitous among weakly integral modular categories as well as anyon models \cite{HNW}, an important open problem is to prove Conjectures \ref{Property_F_conjecture} and \ref{Conjecture_Localizable} for metaplectic anyons.  Some partial results are found in \cite{RWe1}.

We structure the paper as follows: Section 2 contains the basic definitions that we use in subsequent sections, including various notions of dimension.  Section 3 surveys properties of unitary braided fusion categories that are (conjecturally, in some cases) determined by dimension.  Section 4 completes the classification and enumeration of metaplectic modular categories and Section 5 applies this classification to classify certain modular categories of dimension $16m$ with $m$ square-free.
\section{Preliminaries and Dimension Functions}\label{Section: Preliminaries}
   We outline the axioms for the categories we are
interested in, referring the reader to \cite{BKi,ENO1} for further details. 

A modular tensor category is a nondegenerate ribbon fusion category
defined over a subfield $k \subset \mathbb{C}$. We unravel these adjectives
with the following definitions.
\begin{enumerate}
\item A \textbf{monoidal category} is a category with
 a tensor product $\otimes$
and an identity object $\one$ satisfying axioms that guarantee that
the tensor product is associative (at least up to isomorphism) and
that
\begin{equation*}\label{idob}
\one\otimes X\cong X\otimes\one\cong X
\end{equation*}
 for any object $X$. 
\item A monoidal category has left \textbf{duality} if there is a dual
object $X^*$ for each object $X$ and morphisms
$$coev_X: \one \rightarrow X\otimes X^*, \qquad ev_X: X^*\otimes X \rightarrow \one$$
satisfying
\begin{eqnarray*}
(\one_X\otimes ev_X)(coev_X\otimes \one_X)&=&\one_X, \\
(ev_X\otimes \one_{X^*})(\one_{X*}\otimes coev_X)&=&\one_{X^*}.
\end{eqnarray*} Right duality is defined similarly. The duality
allows us to define (left) duals of morphisms too: for any\\ $f\in
\Hom(X,Y)$ we define $f^*\in \Hom(Y^*,X^*)$ by:
$$f^*=
(ev_Y\otimes \one_{X^*})(\one_{Y^*}\otimes f\otimes
\one_{X^*})(\one_{Y^*}\otimes coev_X).$$
A monoidal category that has left and right duality is called \textbf{rigid}.
\item A
\textbf{braiding} in a monoidal category is a natural family of isomorphisms
$$c_{X,Y}: X\otimes Y \rightarrow Y\otimes X$$
satisfying
\begin{eqnarray*}
c_{X,Y\otimes Z}&=&(\one_Y\otimes c_{X,Z})(c_{X,Y}\otimes \one_Z), \\
c_{X\otimes Y,Z}&=&(c_{X,Z}\otimes \one_Y)(\one_X\otimes c_{Y,Z}).
\end{eqnarray*}
\item A \textbf{twist} in a braided monoidal category is a natural family of isomorphisms
$$\theta_X: X \rightarrow X$$
satisfying:
\begin{eqnarray*}
\theta_{X\otimes Y}&=&c_{Y,X}c_{X,Y}(\theta_X\otimes \theta_Y).
\end{eqnarray*}
\item In the presence of a braiding, a twist and duality these structures
are compatible if
\begin{eqnarray*}
\theta_{X^*}=(\theta_X)^*.
\end{eqnarray*}  A braided monoidal category with a twist and a compatible duality is
a \textbf{ribbon} category.

\item An abelian $k$-linear category is \textbf{semisimple} if it has the property that every object
$X$ is isomorphic to a finite direct sum of \emph{simple}
objects--that is, objects $X_i$ with $\End(X_i)\cong k$ satisfying the conclusion of Schur's Lemma:
$$\Hom(X_i,X_j)=0\quad \text{for $i\neq j$}.$$ 
\item In a ribbon abelian $k$-linear category one may define a $k$-linear \textbf{trace} of
endomorphisms.  Let $f\in\End(X)$ for some object $X$.  Set:
\begin{eqnarray*}
\Tr(f)=ev_X c_{X,X^*}(\theta_X f\otimes\one_{X^*})coev_X.
\end{eqnarray*}
Note that this is an element of $\End(\one)\cong k$, and so may be identified with a scalar.
Similarly, since $\theta_X\in\End(X)$ for any object $X$, $\theta_{X_i}$ is a scalar
map (as $X_i$ is simple).  We denote this scalar by $\theta_i$.

\item A semisimple rigid monoidal category is called a \textbf {fusion category} (of \textbf{rank} $r$) if it has finitely many isomorphism classes of simple objects enumerated as $\{X_0=\one,X_1,\ldots,X_{r-1}\}$.  In particular, the monoidal unit $\one$ is simple.  It is known \cite[Prop. 4.8.1]{EGNO} that in a fusion category left and right dualities coincide. 

\item A \textbf{premodular category} is a ribbon fusion category.  A premodular category is \textbf{modular}  if its so-called \emph{$S$-matrix} with entries $$S_{i,j}:=\Tr(c_{X_j,X_i}\circ
c_{X_i,X_j})$$ is invertible. Observe that $S$ is a symmetric
matrix.

\item A \textbf{Hermitian} ribbon fusion category $\mcC$ is equipped with an additive involutive operation $\dag$ on morphisms, $\dag:\Hom(X,Y)\rightarrow\Hom(Y,X)$, that is compatible with $\ot$, composition, braiding ($c_{X,Y}^\dag=c_{X,Y}^{-1}$), twists and rigidity morphisms \cite{TuraevWenzl97}.  In particular $(f,g):=\tr_\mcC(fg^{\dag})$ is a non-degenerate Hermitian form on $\Hom(X,Y)$.  If, in addition, $\dag$ acts on $\mbbC$ by complex conjugation and $(f,g)$ is a positive definite form then we say $\mcC$ is \textbf{unitary}.
\end{enumerate}

\begin{remark}
The term ``modular" comes from the following fact: if we set
$T=(\delta_{i,j}\theta_i)_{ij}$ then the map:
$$\begin{pmatrix}
0 &  -1\\
 1 &  0
\end{pmatrix}\rightarrow S, \begin{pmatrix}
1 &  1\\
 0 &  1
\end{pmatrix}\rightarrow T$$
defines a projective representation of the \emph{modular}
\emph{group} $SL(2,\mbbZ)$.  In fact, by re-normalizing $S$ and $T$ one
gets an honest representation of $SL(2,\mbbZ)$.
\end{remark}

\subsection{Dimension functions}
Let $\mcC$ be a fusion category of rank $r$ with $X_0=\one,X_1,\ldots,X_{r-1}$ a collection of representatives of the distinct isomorphism classes of simple objects.
 The Grothendieck semi-ring $K_0(\mcC)$ of $\mcC$ encodes the fusion rules of $\mcC$: it is the based $\mbbZ_+$-ring with basis the simple isomorphism classes of objects and the operations induced by the direct sum $\oplus$ and the tensor product $\otimes$ \cite{Ostrik}.  A \textit{dimension function} on a fusion category $\mcC$ is a unital ring homomorphism $K_0(\mcC)\rightarrow \mbbC$.  Thus the collection of dimension functions on a fusion category only depends on $K_0(\mcC)$.  
 A fusion category with a \textit{commutative} Grothendieck ring (e.g. a braided fusion category) has exactly as many dimension functions as it has simple isomorphism classes of objects (see e.g. \cite[Theorem 2.3]{BGNPRW2}).  A non-commutative fusion category may admit only the trivial dimension function: $\Vec_G$ for a perfect group $G$ is such an example, since a dimension function for $\Vec_G$ is a linear character of $G$.
 
 We will describe three dimension functions each of a rather distinct nature, and then argue that they are identical in the physically relevant case.
 
 \textbf{FP-dimension}
 The fusion rules $N_{i,j}^k:=\dim\Hom(X_i\ot X_j,X_k)$ supply us with an explicit realization of the left-regular representation of $K_0(\mcC)$ via $X_i\rightarrow N_i$ where $(N_i)_{k,j}:=N_{i,j}^k$.  Indeed, $X_i\ot X_j\cong\bigoplus_k N_{i,j}^kX_k$ so that it suffices to check the corresponding matrix equation (exercise) and then extend to $K_0(\mcC)$ linearly $X\rightarrow N_X$.  The first dimension function is the \textit{Frobenius-Perron (FP) dimension}, defined as $\FPdim(X):=\max \Spec(N_X)$ for simple, i.e. the maximal eigenvalue of the fusion matrix $N_i$ for simple $X$.  The existence of such an eigenvalue follows from the Perron-Frobenius theorem applied to the positive matrix $\sum_{i,j}N_iN_jN_i^{T}$ that commutes with each $N_k$. (exercise, see \cite[Section 8]{ENO1}).  The fact that $\FPdim$ is a dimension function and is uniquely determined by $\FPdim(\one)=1$  and $\FPdim(X_i)>0$ is found in \cite{ENO1}.
 
 \textbf{Asymptotic Dimension}
 In the categorical model for topological phases of matter, the state space of $n$ anyons of type $X$ on a disk with boundary labeled by the vacuum anyon $\one$ is $\Hom(\one,X^{\otimes n})$.  How does $\dim_\mbbC\Hom(\one,X^{\ot n})$ grow with $n$?  If $X$ were simply a $d_X$-dimensional vector space over $\mbbC$ and $\one=\mbbC$ then $\dim_\mbbC\Hom(\one,X^{\ot n})=d_X^n$ for all $n$.  More generally a natural measure of the dimension of an object $X$ in a fusion category is a constant $d_X>0$ such that $\dim_\mbbC\Hom(\one,X^{\ot n})$ grows like $d_X^n$, i.e. the asymptotic dimension of $\Hom(\one,X^{\ot n})$.  Of course it can happen that  $\dim_\mbbC\Hom(\one,X^{\ot n})$ is $0$ for many $n$, but there is a minimal $k>0$ such that $\dim_\mbbC\Hom(\one,X^{\ot kn})>0$ for all $n>0$ (see \cite[Lemma F.6]{DGNO1}), so that we may instead define $d_X>0$ to be a constant such that $\dim_\mbbC\Hom(\one,X^{\ot k(n+1)})/\dim_\mbbC\Hom(\one,X^{\ot kn})\approx (d_X)^k$.  The Perron-Frobenius theorem applied to the fusion matrix $N_X$ implies that $d_X=\FPdim(X)$ (exercise, compute $\dim_\mbbC\Hom(\one,X^{\ot kn})$ in terms of the entries of $N_X$).  Thus the asymptotic dimension coincides with the FP-dimension and in particular is a dimension function.
 
 We illustrate this with two familar examples:
 \begin{example}\label{Example_Fib}
Consider $\mcC = \Fib$ the Fibonacci modular category. There are 2 simple objects $\one, X$ in $\mcC$. Since $X\otimes X = \one \oplus X$ then the fusion matrix is $N_X = \paren{\begin{smallmatrix}
      0 & 1  \\
      1 & 1
    \end{smallmatrix}}$ and it has eigenvalues $\frac{1\pm \sqrt{5}}{2}$. Therefore, $\dim (X) = \frac{1 + \sqrt{5}}{2}$.
    
    The fusion rules of the category $\mcC$ are given by the Fibonacci numbers in the following way $X^{\otimes i} = F(i-1)\one \oplus F(i) X$. Then we can compute:
    $$\lim\limits_{i\to\infty} \frac{\dim \Hom_{\mcC}(X^{\otimes 2i}, \one)}{\dim \Hom_{\mcC}(X^{\otimes 2(i-1)}, \one)} = \lim\limits_{i\to\infty} \frac{F(2i-2)}{F(2i-4)} =\left(\frac{1 + \sqrt{5}}{2}\right)^2.$$
    
\end{example}
\begin{example}\label{ising ex} An Ising modular category $\mathcal{I}^{\nu}$ is any of the $8$ non-integral $4$-dimensional unitary modular categories \cite{DMNO1}.  There are three simple objects $\one$, $\psi$ and $\sigma$.  The fusion matrix for $\sigma$ is $\begin{pmatrix} 0 & 0& 1\\0 &0 &1
\\ 1 & 1 & 0\end{pmatrix}$, so that $\dim{\sigma}=\sqrt{2}$.  Computing, we have $\sigma^{2n}=2^{n-1}(\one\oplus \psi)$ so that $$\frac{\dim \Hom_{\mathcal{I}^\nu}(\sigma^{\otimes 2i}, \one)}{\dim \Hom_{\mathcal{I}^\nu}(\sigma^{\otimes 2(i-1)}, \one)}=\frac{2^{i-1}}{2^{i-2}}=(\sqrt{2})^2.$$
\end{example}
\textbf{Categorical Dimension}
 Now let $\mcC$ be a ribbon fusion category, and $X\in\mcC$.  
Now since $\End(\one)\cong \mbbC$, with basis $\Id_\one$, we may define $\dim(X)$ to be the coefficient of $\Id_\one$ in $\tr_\mcC(\Id_X)$.  In \cite{Turaev92} it is shown that $\dim(-)$ is a dimension function (exercise: show that $\tr_\mcC$ is multiplicative with respect to $\ot$, using naturality of $c$ and $\theta$.).
In fact, one may define a categorical dimension for any spherical fusion category using the pivotal trace, but we will focus on the braided case.

In particular, if $\mcC$ is unitary, then $(Id_X,Id_X)=\tr_\mcC(Id_X)=\dim(X)>0$ for all $X$. Hence for a unitary category $\dim(X)=\FPdim(X)$ for all $X$.

Unless otherwise stated, in the rest of this article we will assume that $\dim(X)=\FPdim(X)$ for all objects.  In many situations (for example if $\FPdim(X)^2\in\mbbZ$ for all simple $X$) it is possible to replace the twists (or equivalently the spherical structure) on a ribbon fusion category by another (unique) choice to ensure that $\dim(X)=\FPdim(X)$ (see \cite{ENO1}).  On the other hand, there are examples \cite{RowellJPAA08} of modular categories whose underlying fusion rules do not admit a unitary ribbon categorification.
The positivity of $\dim$ actually implies a slightly stronger condition:
\begin{theorem} For every object $X$,
$\dim X \geq 1$.
\end{theorem}
\begin{proof}
The dimension $\dim(X)$ is strictly positive and dominates all other eigenvalues of $N_X$ in modulus.  Thus, if $\dim X < 1$, we must have $N_X$ nilpotent since $N_X^n$ would tend to $0$. But  $\dim \Hom(X^{\otimes n}, \one)>0$ for some $n$ by \cite[Lemma F.6]{DGNO1}, so that $N_X^n$ must be non-zero.
\end{proof}

We introduce the following standard notation:
\begin{itemize}
  \item If $X\in\mcC$ has $\dim(X)=1$ we say that $X$ is \textit{invertible}.  If every simple $X\in\mcC$ is invertible, we say $\mcC$ is \textit{pointed}.
  \item If $X\in\mcC$ has $\dim(X)\in \mbbZ$ we say that $X$ is \textit{integral}.  If every simple $X\in\mcC$ is integral, we say $\mcC$ is \textit{integral}.
    \item If $X\in\mcC$ is simple and has $\dim(X)^2\in \mbbZ$ we say that $X$ is \textit{weakly integral}.  If every simple $X\in\mcC$ is weakly integral, we say $\mcC$ is \textit{weakly integral}.  If, moreover, there exists a simple object $X\in\mcC$ with $\dim(X)\not\in\mbbZ$ we say $\mcC$ is \textit{strictly weakly integral}.
\end{itemize}

\begin{remark}\label{formula Xotimesdual}  While invertibility is usually defined as $X\ot X^*\cong \one$, the definition above is equivalent.  Indeed,
let $\mcC$ be a fusion category and $X\in \mcC$. Consider $$G[X] = \{Y\in \mcC \mid Y\otimes X \cong X\}.$$ If $Y\in G[X]$, then $\FPdim Y = 1$.

Moreover, $X\otimes X^* = \sum_{Y\in G[X]} Y + \sum_{\FPdim Z > 1} N^Z_{X, X^*} Z$.
\end{remark}


\subsection{Some basic modular categories}

A significant role is played by pointed modular categories, i.e. modular categories with only invertible simple objects.  The classification of pointed modular categories is well-known (going back essentially to \cite{EM}, also see \cite{DGNO1}): they correspond to pairs $(A,q)$ where $A$ is a finite abelian group and $q$ is a non-degenerate quadratic form $q:A\rightarrow \mbbQ/\mbbZ$ i.e. $q(-a)=q(a)$ and the symmetric bilinear form on $A$ defined by $\sigma(a,b)=q(a+b)-q(a)-q(b)$ is non-degenerate.  We denote such a category by $\mcC(A,q)$.   In \cite{DGNO1} these categories are called \textbf{metric groups}.

  The Semion and Ising categories will also appear in the sequel. The Semion categories, denoted by $\Sem$, are modular categories with the same fusion rules as
  $\Rep\paren{\mbbZ_{2}}$, see
  \cite{RSW} for details.  
%
  The Ising categories introduced in Example \ref{ising ex} are rank 3 modular categories with simple objects:
  $\one$, $\psi$ (a fermion), and $\s$ (the Ising anyon) of dimension $\sqrt{2}$.
  The key fusion rules are $\ps^{2}=\one$ and
  $\s^{2}=\one+\ps$; while the modular datum can be found in \cite{RSW}.  There are exactly $8$ inequivalent Ising categories.  They are the smallest examples of \textbf{(generalized) Tambara-Yamagami categories}, \textit{i.e.} non-pointed fusion categories with the property that the tensor product of any two non-invertible simple objects is a direct sum of invertible objects.  In fact, by \cite{Nat} any \textit{modular} generalized Tambara-Yamagami category is a (Deligne) product of an Ising category and a pointed modular category.

\subsubsection{Balancing}
Standard arguments show that the entries of the $S$-matrix are
determined by the dimensions, the fusion rules and the
twists on the simple objects, giving the following
extremely useful balancing relation (see \cite{BKi}):
\begin{eqnarray}\label{sformula}
S_{i,j}=\frac{1}{\theta_i\theta_j}\sum_k N_{i^*,j}^k \dim(X_k) \theta_k.
\end{eqnarray}

\subsection{Centralizers}
  For $\mcD\subset \mcC$ premodular categories the
  \textbf{centralizer} of $\mcD$ in $\mcC$ is denoted by $\mcZ_{\mcC}\paren{\mcD}$
  and is generated by the objects $\lcb X\in\mcC\mid S_{X,Y}=d_{X}d_{Y}\;\forall
  Y\in\mcD\rcb$ (see \cite{Brug1}. 
  The category $\mcZ_{\mcC}\paren{\mcC}$ is called the \textbf{M\"{u}ger center} and is often denoted
  by $\mcC'$. Note that a useful characterization of a simple object being
  outside of the M\"{u}ger center is that its column in the $S$-matrix is
  orthogonal to the first column of $S$.
If
  $\mcC'=\Vec$, the category of finite-dimensional vector spaces, then $\mcC$ is a modular category whereas if
  $\mcC'=\mcC$, then $\mcC$ is called \textbf{symmetric}. 
  
  \subsubsection{Symmetric fusion categories}
  
  Symmetric fusion
  categories are classified in terms of group data:
  \begin{theorem}[\cite{D1}]
    If $\mcC$ is a symmetric fusion category, then there exists a finite group $G$ such that
    $\mcC$ is equivalent to the super-Tannakian category $\Rep(G,z)$ of
    super-representations (i.e. $\mbbZ_2$ graded) of $G$ where $z\in Z(G)$ is a distinguished central element with $z^2=1$
 acting as the parity operator.
  \end{theorem}
  Observe that $\mcC\cong\Rep(G,z)$ with $z=e\in G$ if and only if the $\mbbZ_2$-grading is trivial so that $\Rep(G,z)=\Rep(G)$.  In this case we say $\mcC$ is \textbf{Tannakian}, and otherwise we say it is non-Tannakian.  The smallest non-Tannakian symmetric category is $\Rep(\mbbZ_2,1)$ which we will denote by $\sVec$ when equipped with the (unique) structure of a ribbon category with $\dim(\chi)=1$ for a non-trivial object $\chi$.
 
 An invertible
  object in $\mcC$ that generates the Tannakian category $\Rep\paren{\mbbZ_{2}}$ is
  a {boson} while an invertible object in $\mcC$
  generating $\sVec$ is a {fermion}.  Bosons are useful through a process known as de-equivariantization  which we will discuss shortly.  A modular category $\mcC$ with a distinguished fermion $f\in\mcC$ is called a \textbf{spin modular} category (see, eg. \cite{BGNPRW2}).

\subsection{Gradings}
  A \textbf{grading} of a fusion category $\mcC$ by a finite group $G$ is a
  decomposition of the category as direct sum $\mcC = \oplus_{g\in G} \mcC_g$,
  where the components are full abelian subcategories of $\mcC$ indexed by the
  elements of $G$, such that the tensor product maps
  $\mcC_g\times \mcC_h$ into $\mcC_{gh}$.  The \textbf{trivial component}
  $\mcC_e$ (corresponding to the unit of the group $G$) is a fusion subcategory
  of $\mcC$. The grading is called \textbf{faithful} if $\mcC_g \neq 0$, for all
  $g\in G$, and in this case all the components are equidimensional with
  $\abs{G}\dim \mcC_g = \dim\mcC$.
  

  Every fusion category $\mcC$ is faithfully graded by the universal grading group $\mcU\paren{\mcC}$
   and every faithful grading of
  $\mcC$ is a quotient of $\mcU\paren{\mcC}$. Furthermore, the trivial component
  under the universal grading is the \textbf{adjoint subcategory} $\mcC_{\ad}$, the full fusion subcategory generated by the objects $X\otimes X^*$, for $X$ simple.  
  The universal grading was first studied in \cite{GN2}, and it was shown that
  if $\mcC$ is modular, then $\mcU\paren{\mcC}$ is canonically isomorphic to the
  character group of the group $G(\mcC)$ of isomorphism classes of invertible objects in
  $\mcC$. Another useful fact is that $\mcC_{ad}' = \mcC_{pt}$, the subcategory generated by the invertible objects \cite{GN2}.
  
  When $\mcC$ is a weakly integral fusion category there is another useful grading called the \textbf{GN-grading} first studied in \cite{GN2}:
  \begin{theorem}\cite[Theorem 3.10]{GN2}
    Let $\mcC$ be a weakly integral fusion category. Then there is an elementary
    abelian $2$-group $E$, a set of distinct square-free positive
    integers $n_x$, $x \in E$, with $n_0 = 1$, and a faithful grading $\mcC =
    \oplus_{x\in E} \mcC(n_x)$ such that $\dim(X) \in \mathbb Z \sqrt{n_x}$
    for each $X \in \mcC(n_x)$.
  \end{theorem}
The trivial component of the $GN$-grading is the integral subcategory \textbf{$\mcC_{\intcat}$}.

\subsection{Equivariantization}
Our approach to classification relies upon \textit{equivariantization} and its inverse functor 
  \textit{de}-\textit{equivariantization} (see
  for example \cite{DGNO1}) which we now briefly describe.
  An \textit{action} of a finite group $G$ on a fusion category $\mcC$ is a
  strong tensor functor $\rho: \underline{G}\to \End_{\otimes}(\mcC)$.  The
  \textbf{$G$-equivariantization} of the category $\mcC$ is the category $\mcC^G$
  of $G$-equivariant objects and morphisms of $\mcC$. When $\mcC$ is a fusion
  category over an algebraically closed field $\mbbK$ of characteristic $0$, the
  $G$-equivariantization, $\mcC^G$, is a fusion category with  $\dim \mcC^G = |G| \dim \mcC$. The fusion
  rules of $\mcC^G$ can be determined in terms of the fusion rules of the
  original category $\mcC$ and group-theoretical data associated to the group
  action \cite{BuN1}. De-equivariantization is the inverse to equivariantization. Given a fusion category $\mcC$ and a
  Tannakian subcategory $\Rep G\subset\mcC$, consider the algebra $A = \Fun\paren{G}$ of
  functions on $G$. Then $A$ is a commutative algebra in $\mcC$.
  The category of $A$-modules on $\mcC$, $\mcC_{G}$, is a
  fusion category called a \textbf{$G$-de-equivariantization} of
  $\mcC$.  We have
  $\abs{G}\dim \mcC_G = \dim \mcC$, and there are
  canonical equivalences $\paren{\mcC_G}^G \cong \mcC$ and $\paren{\mcC^G}_G
  \cong \mcC$. An important property of the de-equivariantization is that if the
  Tannakian category in question is $\mcC'$, then $\mcC_{G}$ is modular and is
 called the \textbf{modularization} of $\mcC$ \cite{Brug1,M5}.

\subsection{Gauging}
Two processes that we employ in our analysis are gauging and de-gauging.  First let us describe de-gauging. Let $\mcC$ be modular and $\Rep(G)\cong \mcD\subset\mcC$ a Tannakian subcategory (here a Tannakian category is a symmetric braided fusion category equivalent to $\Rep(G)$ for some finite group $G$).  The $G$-de-equivariantization $\mcC_G$ of $\mcC$ is a faithfully $G$-graded category (in fact, a braided $G$-crossed category) with trivial component $[\mcC_G]_e$ a modular category of dimension $\dim(\mcC)/|G|^2$ (see \cite{DMNO1}).  $[\mcC_G]_e$ is the \textbf{$G$-de-gauging} of $\mcC$.  The reverse process, $G$-gauging, is more complicated.  Here one starts with a modular category $\mcB$ and an action of a finite group $G$ by braided tensor autoequivalences: $\rho:G\rightarrow \Aut_{\ot}^{br}(\mcB)$.  A  \textbf{$G$-gauging} of $\mcB$, when it exists, is a new modular category obtained by first constructing a $G$-graded fusion category $\mcD$ with trivial component $\mcD_e=\mcB$ and then equivariantizing $\mcD^G$.  There are obstructions to the existence of a gauging, and when the obstructions vanish there can be many $G$-gaugings.

Here is a key example: consider the pointed (cyclic) modular category $\mcC(\mbbZ_N,q)$.  The elements of the group $\Aut_{\ot}^{br}\mcC(\mbbZ_N,q)$ are simply those $\phi\in\Aut(\mbbZ_N)$ that preserve $q$ \cite{DN}. 
The following is presumably well-known, but its proof is elementary:
\begin{prop}
The \textbf{particle-hole symmetry} $\phi(a)=-a$ is the only non-trivial element of  $\Aut_{\ot}^{br}\mcC(\mbbZ_N,q)$.
\end{prop}
\begin{proof} By the Chinese remainder theorem $\mcC(\mbbZ_N,q)$ factors as a product of cyclic modular categories of prime power dimension, so we may assume that $N=p^k$.  For $p\neq 2$ the two inequivalent quadratic forms $q:\mbbZ_{p^k}\rightarrow \mbbQ/\mbbZ$ are $q_u(a)=\frac{ua^2}{p^k}$ for $u=\pm 1$.  An automorphism of $\mbbZ_{p^k}$ is given by $\phi_x(a)=xa$ for some $x\in\mbbZ_{p^k}^*$.  So $q_u\circ \phi_x=q_u$ implies that $x^2\equiv 1\pmod{p^k}$, which has only $x=\pm 1$ as solutions.  For $p=2$ there are 4 inequivalent quadratic forms (see, e.g. \cite{GalindoJaramillo}): $q_u(a)=\frac{ua^2}{2^{k+1}}$ for $u\in\mbbZ_8^*$.  The same computation as above shows that $q_u\circ\phi_x=q_u$ only if $x^2\equiv 1\pmod{2^{k+1}}$ with $x\in\mbbZ_{2^k}^*$.  Although $x^2  \equiv\pmod{2^{k+1}}$ has 4 solutions, there are only two inequivalent solutions modulo $2^k$, namely $x=\pm 1$.
\end{proof}
 Furthermore, it can be shown that the obstructions vanish in this case, so the action $\rho:\mbbZ_2\rightarrow \Aut_{\ot}^{br}\mcC(\mbbZ_N,q)$ defined by $\rho(1)=\phi$ can be gauged.  

\section{Properties Determined by Dimension}\label{Section: Properties}
In this section we assume that $\mcC$ is a unitary braided fusion category, and summarize some results showing that several important properties of any anyon are detected by the dimension of the corresponding simple object.

Important examples of unitary braided fusion categories are obtained from quantum groups at roots of unity (see \cite{RowellSurvey} for a survey).  Here we briefly outline the construction, mainly for notational purposes.

\begin{enumerate}
\item Let $\mfg$ be a simple Lie algebra and $q=e^{\pi i/\ell}$ be a root of unity with $\ell>\check{h}$ (dual Coxeter number of $\mfg$) and $\ell\in 2\mbbZ$ for types $B,C,F$ and $\ell\in3\mbbZ$ for type $G$.
\item The representation category of quantum group $U_q\mfg$ is a non-semisimple ribbon category.  A quotient by the tensor ideal of negligible morphisms (essentially the radical of the trace) of this category gives a braided fusion category $\mcC(\mfg,\ell)$.
\item The integer $k=(\ell-\check{h})/m$ (where $m=2$ for types $B,C,F$ and $m=3$ for type $G$ and $m=1$ otherwise) is called the \textit{level}, and for classical types $A,B,C$ and $D$ we adopt the alternative notation for $\mcC(\mfg,\ell)$ is $G(N)_k$ for $G=SU,SO$ or $Sp$ of dimension $N$.
\end{enumerate}

The braid group on $n$ strands $\mcB_n$ generated by $\sigma_1,\ldots,\sigma_{n-1}$ satisfying
\begin{enumerate}
\item $\sigma_i\sigma_{i+1}\sigma_i=\sigma_{i+1}\sigma_i\sigma_{i+1}$ for $i\leq n-2$
\item $\sigma_{i}\sigma_j=\sigma_{j}\sigma_i$ for $|i-j|>1$
\end{enumerate} plays an important role in topological phases of matter: particle exchange induces a representation of $\B_n$ on the state space $\Hom(\one,X^{\ot n})$ for any object $X$.  More generally we obtain a homomorphism of $\mcB_n$ into $\Aut(X^{\ot n})$, which is often lifted to a homomorphism of the group algebra $\rho_X:\mbbC\mcB_n\rightarrow\End(X^{\ot n})$.  Moreover, $\mcH_{n}:=\bigoplus_i\Hom(X_i,X^{\ot n})$ is a faithful $\End(X^{\ot n})$-module and hence 
we will sometimes abuse notation and regard $\rho_X$ as a $\B_n$-representation on $\mcH_{n}$.

\subsection{Abelian Anyons}
We say that $X\in\mcC$ (or the corresponding anyon) is \textit{non-abelian} if the $\B_n$-representation $(\rho_X,\mcH_{n})$ has non-abelian image for all $n\geq 3$.  This is the first property of $X\in\mcC$ determined by dimension:
\begin{theorem}[\cite{RWJPA}]
If $X\in\mcC$ is simple and $\dim X > 1$ then $X$ is non-abelian. 
\end{theorem}
It is clear that any invertible $Z\in\mcC$ is abelian: $\dim\End(Z^{\ot n})=1$ in this case so $\rho_Z(\B_n)$ acts by scalars.  Thus we have a complete characterization of (simple) abelian anyons as those corresponding to invertible objects.
A \textbf{boson} $b$ is a particular kind of abelian anyon, characterized by $b^{\otimes 2}\cong\one$ and $c_{b,b}=\Id_{b\otimes b}$ whereas a \textbf{fermion} $f$ has $f^{\otimes 2}\cong \one$ but $c_{f,f}=-\Id_{f\otimes f}$.  Alternatively, in a unitary ribbon category, an object $X$ with $X^{\ot 2}\cong\one$ is a boson if $\theta_X=1$ and a fermion if $\theta_X=-1$.
\subsection{Localizable Anyons}
The $\B_n$ representations $\rho_X$ are somewhat complicated--they exhibit a \textit{hidden locality} \cite{FKW}, but are not explicitly local--rather the action of $\rho_X(\sigma_i)$ act non-trivially on the entire space $\mcH_n$.  Explicitly local representations of $\B_n$ can be obtained from solutions $R\in\Aut(V^{\ot 2})$ to the Yang-Baxter equation on a vector space $V$:
\begin{equation}
    (R\ot Id_V)(Id_V\ot R)(R\ot Id_V)=(Id_V\ot R)(R\ot Id_V)(Id_V\ot R).
\end{equation}  Such a pair $(R,V)$ is called a \textit{braided vector space}.  From a braided vector space we obtain a representation $\rho^R$ of $\B_n$ on $V^{\ot n}$ via $$\sigma_i\rightarrow Id_V^{\ot (i-1)}\ot R\ot Id_V^{\ot (n-i-1)}.$$
We say that $X\in\mcC$ is localizable  if there is a  braided vector space $(R,V)$
and injective
algebra maps $\tau_n:\mbbC\rho_X(\B_n)\rightarrow \End(V^{\ot n})$ such that,
for all $n$, $\rho^R=\tau_n\circ \rho_X$.  That is, the following diagram commutes:
 $$\xymatrix{ \mbbC\B_n\ar[d]^{\rho_X}\ar[dr]^{\rho^R} \\
\mbbC\rho_X(\B_{n})\ar[r]^{\tau_n} & \End(V^{\ot n})}$$

\begin{example}
Consider $\mcI^\nu$ an Ising category (see Example \ref{ising ex}) and the object $\sigma$ of dimension $\sqrt{2}$.
For 
$R = \frac{1}{\sqrt{2}} \paren{\begin{smallmatrix}
      1 & 0 & 0 & 1  \\
      0 & 1 & 1 & 0 \\
      0 & -1 & 1 & 0 \\
      -1 & 0 & 0 & 1
    \end{smallmatrix}}$ the braided vector space $(R,\mbbC^2)$ provides a localization of $\sigma$ \cite{FRW}.
    
\end{example}

Two slightly less restrictive notions of localizability are studied in \cite{GHR}, namely $(k,m)$-generalized localizations and quasi-localizations.  Under some assumptions, localizability is known to be determined by dimension:
\begin{theorem}[\cite{RW_localization,GHR}] Let $X\in\mcC$ be a simple object.
If $\rho_X(\mathbb{C} \B_n)=\End(X^{\otimes n})$ for $n\geq 2$ and $X$ is (generalized or quasi-)localizable then $\dim (X)^2\in \mathbb Z$.
\end{theorem}

It is believed that this relationship holds more generally:
\begin{conjecture}[\cite{RW_localization,GHR}]\label{Conjecture_Localizable}
Any simple $X\in\mcC$ is (generalized or quasi-)localizable if, and only if $\dim(X)^2\in\mbbZ$.
\end{conjecture}

The Gaussian Yang-Baxter operators described in \cite{GR} provide localizations of the generating (fundamental spinor) objects in $SO(N)_2$.

\subsection{Universal Anyons}

A given topological model for quantum computation is called (braiding-only) \textit{universal} if any unitary operator can be efficiently approximated up to a phase by braiding anyons \cite{FLW}.  We say that $X\in\mcC$ is \textit{braiding universal} if $\rho_X(\B_n)$ is dense in $SU(W)$ for each irreducible subrepresentation $W$ of $\mcH_n$.  The first step to verify universality is to check that $\rho_X(\B_n)$ is infinite, and very often (see \cite{FLW,LRW}) this is sufficient.  It is therefore of the utmost importance to determine when $\rho_X(\B_n)$ is finite.
\begin{definition}\label{Property_F_definition}
An object $X\in\mcC$ has \textit{property F} if the image $\rho_X(\B_n)$ is finite. 

A braided fusion category $\mcC$ has \textit{property F} if the associated braid
group representations on the centralizer algebras $\End_{\mcC}(X^{\otimes n}
)$ have finite image for
all $n$ and all objects $X$.
\end{definition}

\begin{conjecture}\cite{NR1}\label{Property_F_conjecture}
A braided fusion category $\mcC$ has property F if, and only if,  $\FPdim(\mcC)\in \mathbb Z$ (i.e. $\mcC$ is weakly integral).
\end{conjecture}

Significant progress has been made towards proving this conjecture:
\begin{enumerate}
\item Every $X\in \Rep(D^{\omega}(G))$ has property F \cite{ERW_prop_F_gt}.
\item Every weakly integral quantum group category has property F, including $SU(2)_2$, $SU(2)_4$, $SU(3)_3$ and $SO(N)_2$ (see \cite{FLW,LRW,LR,RWe1}).
\item The standard tensor generators of every non-weakly integral quantum group category do not have property F.  (see \cite{FRW,LRW,RowellRUMA}).
\end{enumerate}

From the discussion above, it is clear that weakly integral braided fusion categories (conjecturally) have interesting properties.  The main examples we have of such categories are metaplectic categories, which are closely related to $SO(N)_2$.  We will explore these categories in some detail in the next sections.

\section{Metaplectic modular categories}

\begin{defn}
    A \textbf{metaplectic modular category (of dimension $4N$)} is a modular category
    $\mcC$ with positive dimensions that is Grothendieck equivalent to $SO(N)_2$, for some integer $N  \geq 2$.
  \end{defn}
  
 These fusion rules of metaplectic categories differ in important ways depending on the value of $N\pmod{4}$.   Metaplectic categories for $N$ odd were defined in \cite{HNW} and studied in \cite{ACRW1}, while the case $N\equiv 2\pmod{4}$ can be found in \cite{BPR} where the term ``even metaplectic" is used to describe metaplectic modular categories of dimension $4N$ with $N$ even.  We will simplify the terminology and call them all metaplectic (of dimension $4N$), and specify the value of $N\pmod{4}$ when necessary.

The ubiquity of metaplectic categories among weakly integral modular categories as well as their application to topological phases of matter motivate their classification.  The complete classification of weakly integral modular categories of dimension $2^n m$, for $n\in \{0, 1, 2, 3\}$ with $m$ a square-free odd integer has been given, employing the classification of metaplectic modular categories of dimension $4N$ with $N$ odd or $N\equiv 2\pmod{4}$ \cite{BGNPRW1,BPR}.  In \cite{DN} a general description of modular categories of dimension $p^nm$ with $m$ square-free is given in terms of de-equivariantizations. Here we consider metaplectic modular categories of dimension $4N$, recalling the known results for $4\nmid N$ with new results in the case $4\mid N$.

\subsection{Metaplectic modular categories of dimension $4N$ with $4\nmid N$}\label{Subsection: Metaplectic}

The metaplectic modular categories of dimension $4N$ with $4\nmid N$ have been studied in \cite{ACRW1,BPR}.  Most of the fusion rules for such a category were given in \cite{NR1} with more complete details in \cite{HNW,BPR}.  For $N$ odd there are $2$ invertible objects, $\frac{N-1}{2}$ simple objects of dimension $2$ and $2$ simple objects of dimension $\sqrt{N}$.  All simple objects are self-dual in this case.  In the case $N \equiv 2\pmod{4}$ we have $4$ invertible objects (including one pair of non-self-dual objects), $4$ (non-self-dual) simple objects of dimension $\sqrt{N/2}$, and $\frac{N}{2}-1$ simple objects of dimension $2$ (all of which are self-dual).  

We have a complete classification of these categories:

\begin{theorem}\cite{ACRW1,BPR} \label{Theorem: Metaplectic main result} If $\mcC$ is a metaplectic category of dimension $4N$, with $4\nmid N$, then $\mcC$ is a gauging of the particle-hole symmetry of a
    $\mbbZ_{N}$-cyclic modular category $\mcC(\mbbZ_N,q)$. Moreover, for $N = p_1^{k_1} \cdots
    p_r^{k_{r}}$, with $p_i$ distinct primes (where if $p_1=2,k_1=1$), there are exactly $2^{r+1}$
    many inequivalent such metaplectic modular categories.
    
  \end{theorem}

\subsection{Metaplectic categories of dimension $4N$, with $4\mid N$}
  \label{Subsection: Even Metaplectic Classification for N even}

 The modular category $SO(N)_2$ with $N\equiv 0\pmod{4}$ corresponds to Lie type $D_k$ with $2k=N$, has
  rank $k+7$ and dimension $4N$ \cite{NR1}.  The simple objects have dimension $1,2$ and $\sqrt{k}$.
  
The simple objects of $SO(N)_2$ are naturally labeled by $\mathfrak{so}_{N}$ weights $\mu$ with $\mu^1+\mu^2\leq 2$, i.e. the sum of the first two coordinates can be at most $2$.  We will provide them with less cumbersome labels, after identifying their weights in terms of the fundamental weights $\l_j= (1, \ldots, 1, 0, \ldots 0)$ with $j$ 1s for $j \leq k-2$, $\l_{k-1} = \frac{1}{2}(1, \ldots,1, -1)$, and $ \l_{k} = \frac{1}{2}(1, \ldots, 1)$.  Setting $r=\frac{k}{2}-1$, we have simple objects as follows:
\begin{itemize}
    \item  objects with weights $\mathbf{0},2\l_{1},2\l_{k-1},2\l_{k}$ will be denoted $\one,fg,f,g$, 
    \item  simple objects with weights $\l_2,\ldots,\l_{k-2}$ will be denoted $X_0,\ldots,X_{r-1}$,
    \item simple objects with weights $\l_1,\l_3,\ldots,\l_{k-3},\l_{k-1}+\l_{k}$ will be denoted $Y_0,\ldots,Y_r$,
    \item simple objects with weights $\l_{k-1},\l_1+\l_k$ will be denoted $V_1,V_2$, and
    \item simple objects with weights $\l_{k},\l_1+\l_{k-1}$ will be denoted $W_1,W_2$.
\end{itemize}

 All simple objects are self-dual, and the $X_i$ and $Y_i$ have dimension $2$, while $V_i,W_i$ have dimension $\sqrt{k}$.  
 For $k>2$ the key fusion rules are as follows, where we abuse notation and write $=$ for $\cong$:

  \begin{itemize}
  \item $f^{\ot 2}=g^{\ot 2}=\one$, $f\ot X_i=g\ot X_i=X_{r-i-1}$ and $f\ot Y_i=g\ot Y_i=Y_{r-i}$
  \item $g\ot V_1=V_2, f\ot V_1=V_1$ and $f\ot W_1=W_2, g\ot W_1=W_1$
  \item $V_1^{\ot 2}=\one\oplus f\oplus\bigoplus_{i=0}^{r-1} X_i$
  \item $W_1^{\ot 2}=\one\oplus g\oplus\bigoplus_{i=0}^{r-1} X_i$
  \item $W_1\ot V_1=\bigoplus_{i=0}^r Y_i$
  \item $X_i\ot X_j=\begin{cases} X_{i+j+1}\oplus X_{j-i-1} & i<j\leq\frac{r-1}{2}\\ \one\oplus fg\oplus X_{2i+1} & i=j<\frac{r-1}{2}\\ \one \oplus f\oplus g\oplus fg & i=j=\frac{r-1}{2}<r-1\end{cases}$
  \item $Y_i\ot Y_j=\begin{cases} X_{i+j}\oplus X_{j-i-1} & i<j\leq\frac{r}{2}\\ \one\oplus fg\oplus X_{2i} & i=j\leq\frac{r-1}{2}\\ \one\oplus f\oplus g\oplus fg &i=j=\frac{r}{2}.\end{cases}$
 
  \end{itemize}

Notice that all other fusion rules may be derived from the above by tensoring with $f$ or $g$ as needed.  For example $V_1\ot V_2=g\ot V_1^{\ot 2}=f\oplus fg\oplus\bigoplus_{i=0}^{r-1}X_i$.

For $k=2$, i.e. $SO(4)_2$ we have $9$ simple objects: $\one,f,g,fg,Y_0,V_1,V_2,W_1,W_2$ and the fusion rules are the same as $SU(2)_2\boxtimes SU(2)_2$.  The applicable fusion rules above still hold.  Indeed, by \cite[Corollary B.12]{DGNO1} any such category is equivalent to a Deligne product of Ising-type modular categories (of which there are 8).  

Recall that a metaplectic category of dimension $4N$ with $4\mid N$ is any modular category $\mcC$ with the same fusion rules as above.  Fix such a category $\mcC$ and label the simple objects as above. The objects $\one,f,g$ and $fg$ are invertible, and their isomorphism classes form the group $\mbbZ_2\times\mbbZ_2$ under tensor product.  In particular the universal grading group $\mcU(\mcC)\cong\mbbZ_2\times\mbbZ_2$ with graded components labeled as follows:
\begin{enumerate}
    \item The adjoint category $\mcC_0$: with simple objects $\one,fg,f,g$ and $X_0,\ldots,X_{r-1}$
    \item $\mcC_{(1,1)}$ with simple objects $Y_0,\ldots,Y_r$
    \item $\mcC_{(1,0)}$ with simple objects $V_1,V_2$
     \item $\mcC_{(0,1)}$ with simple objects $W_1,W_2$.
\end{enumerate}
 \begin{lemma}
    \label{even_boson}
    Let $\mcC$ be a metaplectic modular category of dimension $4N$ with $4\mid N$, with simple objects labeled as above.  Then:
    \begin{enumerate}
        \item $fg$ centralizes every object in $\mcC_0\oplus \mcC_{(1,1)}$.
        \item $fg$ is a boson, i.e. $\theta_{fg}=1$.
    \end{enumerate}
  \end{lemma}
  \begin{proof}

 Observe that since the pointed subcategory $\mcC_{pt}$ with simple objects $\one,f,g,fg$ is a subcategory of $\mcC_{ad}=\mcC_0$ and $\mcC_{pt}=\mcC_{ad}^\prime$ we see that $\mcC_{pt}$ is symmetric, i.e. self-centralizing. Moreover,  since $fg\in\mcC_{pt}=\mcC_{ad}^\prime$ it is clear that $fg$ centralizes $\mcC_0$.  For $Y_i\in\mcC_{(1,1)}$ we have $fg\ot Y_i=Y_i$ and the balancing equation gives: $$\theta_{Y_i}\theta_{fg}S_{Y_i,fg}=\theta_{Y_i}\dim(Y_i)$$ so that $S_{Y_i,fg}=\dim(fg)\dim(Y_i)$ if and only if $\theta_{fg}=1$.  If $k>2$ (so $r>0$) there is a 2-dimensional object $X_0\in\mcC_{ad}$ with $X_0\ot fg=X_0$ and $S_{X_0,fg}=2=\dim(X_0)\dim(fg)$, so that the balancing equation implies $\theta_{fg}=1$, and the result follows.
 
 For the case $k=2$, $\mcC$ is a product of Ising-type categories, and there is a labeling ambiguity among $f,g$ and $fg$.  Precisely one of these is a boson, since the non-trivial invertible object in any Ising category is a fermion.  So we may assume $fg$ is the boson, and the result follows as above.
  %
  \end{proof}

In particular, the subcategory $\langle fg\rangle$ generated by $fg$ is equivalent, as a symmetric ribbon category, to $\Rep(\mbbZ_2)$.

We also have the following lemma describing the other invertible objects.
   
  \begin{lemma}
Let $\mathcal C$ be  a metaplectic modular category of dimension $4N$ with $4\mid N$. If $8\mid N$ then  all the non-trivial invertible objects of $\mathcal C$ are bosons, and if $8\nmid N$ two are fermions (and one is a boson).
  \end{lemma}
  \begin{proof}
    From the description above, we know that there are $r = \frac{N}{4}-1$  $2$-dimensional simple objects in the adjoint component $\mathcal C_0$, and $r+1$ $2$-dimensional simple objects in $\mcC_{(1,1)}$. Moreover, the pointed subcategory is symmetric and has fusion rules like $\mbbZ_2\times\mbbZ_2$.
    
    Since $\mathcal C_0' = \mcC_{pt}$, if one of the invertibles is a fermion, the action of tensoring with that fermion must be fixed-point free on $\mathcal C_0$. Then, the number of $2$-dimensional simple objects in $\mcC_0$, $r$, must be even in this case. In particular, $N= \dim \mathcal C = 4 + 4r$ is not divisible by $8$.  So if $8\mid N$, $r$ is odd and so there are no fermions.  
    
On the other hand, if $8\nmid N$ we must have $r$ even, and so $Y_{\frac{r}{2}}^{\ot 2}=\one\oplus f\oplus g\oplus fg$.  The balancing equation then gives, for example, $-2\theta_{Y_{\frac{r}{2}}}\theta_f=2\theta_{Y_{\frac{r}{2}}}$, so that $\theta_f=\theta_g=-1$.
      \end{proof}

\subsection{Proof of Theorem \ref{thm:particleHoleAnalog}}

We can now prove a partial analogue to Theorem \ref{Theorem: Metaplectic main result} for metaplectic modular categories of dimension $4N$ with $4\mid N$.  One exception is the case $N=4$.

\begin{theorem} \label{thm:particleHoleAnalog}
    If $\mcC$ is a metaplectic modular category of dimension $4N>16$ with $4\mid N$ then the de-equivariantization $\mcD:=\mcC_{\mbbZ_2}$ by $\langle fg\rangle=\Rep(\mbbZ_2)$ is a  generalized
    Tambara-Yamagami category of dimension $2N$, and, the trivial
    component $\mcD_0:=[\mcC_{\mbbZ_2}]_0\cong\mcC(\mbbZ_N,q)$ is a pointed cyclic modular category.  Moreover, $\mcC$ is obtained from $\mcC(\mbbZ_N,q)$ via a $\mbbZ_2$-gauging of the particle-hole symmetry.
  \end{theorem}
  \begin{proof}
We continue with the notation and labeling as above, with $\one,f,g,fg,X_i,Y_j$ with $0\leq i\leq r-1$ and $0\leq j\leq r$ where $r=\frac{k}{2}-1$ and $N=2k$.

    As we noted previously, by \lemmaref{even_boson}, $\langle
    gf\rangle\cong\Rep\(\mbbZ_2\)$ is a Tannakian subcategory of $\mcC$. In
    particular we have the de-equivariantization functor  $F:\mcC\rightarrow\mcC_{\mbbZ_2}$, the image of which is a
    braided $\mbbZ_2$-crossed fusion category of dimension $2N$ \cite{DGNO1}. In particular,
    the trivial component $\mcD_0$ of $\mcD=\mcC_{\mbbZ_2}$ is modular of dimension $N$ \cite[Proposition 4.56(ii)]{DGNO1}.

    Since $gf$ fixes the 2-dimensional objects $X_i$, $Y_i$ of $\mcC$, their images under $F$ give rise to $k-1$ inequivalent pairs $X_i^{(1)},X_i^{(2)},Y_i^{(1)},Y_i^{(2)}$  of distinct invertible objects in
    $\mcD$. On
    the other hand, $gf\otimes \one=gf$ and $g\otimes gf=f$, so that $F(gf)=F(\one)=\one_\mcD$ and $F(g)=F(f)=Z$ are distinct invertible objects in $\mcD$.  In all, we have $2(k-1)+2=N$ invertible objects, all of which are in the trivial component $\mcD_0$.  Thus $\mcD_0\cong\mcC(A,q)$ for some abelian group of order $N$.
    
    Since the $4$ objects in $\mcC_{(0,1)}\cup\mcC_{(1,0)}$ form two orbits under ${\_}\ot fg$, we see that $\mcD_1$ has two
    non-integral simple objects which each have dimension $\sqrt{N/2}$. In particular,
    $\mcD$ is generalized Tambara-Yamagami category \cite{Lip1}.
    
    It remains to verify that the classes of simple objects in $\mcD_0$ form a \emph{cyclic} group, which we do following the inductive idea of \cite{ACRW1}.
    

    First note that $Y_{0}\otimes Y_{0}=\one\oplus fg\oplus X_{0}$. Since  $F(\one)=F(fg)=\one_\mcD$ is the 
  the trivial object under the de-equivariantization, we
  have:
  \begin{align*} F(Y_0^{\ot 2})=
    \paren{Y_{0}^{\paren{1}}\oplus Y_{0}^{\paren{2}}}^{\otimes
    2}=2(Y_0^{(1)}\ot Y_0^{(2)})\oplus \left(Y_0^{(1)}\right)^{\ot 2}\oplus\left(Y_0^{(2)}\right)^{\ot 2} =\one_{\mcD}\oplus X_{0}^{\paren{1}}\oplus X_{0}^{\paren{2}}
  \end{align*}

 We will show that the class of $Y_{0}^{(1)}$ generates $\left[\mcC_{\mbbZ_2}\right]_0$. Examining multiplicities we see that we may assume (using the labeling ambiguity $X_0^{(j)}$, $j=1,2$) 
  $Y_{0}^{\paren{1}}\otimes Y_{0}^{\paren{2}}=\one_{\mcD}$ while
  $Y_{0}^{\paren{j}}\otimes Y_{0}^{\paren{j}}=X_{0}^{\paren{j}}$.

  Proceeding in a similar way with $Y_0\ot X_0=Y_0\oplus Y_1$ we must match the $4$ simple objects $Y_0^{(i)}\ot X_0^{(j)}$ for $1\leq i,j\leq 2$ with the four simple objects $Y_a^{(b)}$ for $a=0,1$ and $b=1,2$.  Now since $Y_0^{(2)}=(Y_0^{(1)})^*$ we must have  
  $$Y_{0}^{\paren{1}}\otimes
  X_{0}^{\paren{1}}\oplus Y_{0}^{\paren{2}}\otimes
  X_{0}^{\paren{2}}=Y_{1}^{\paren{1}}\oplus Y_{1}^{\paren{2}}.$$ This is again a labeling ambiguity so we may define, without loss of generality, $Y_{0}^{\paren{j}}\otimes
  X_{0}^{\paren{j}}=Y_{1}^{(j)}$ for $j=1,2$.

Now notice that for $n\leq 2r$ the tensor power $Y_0^{\ot n}$ contains exactly one simple object that has not appeared in lower tensor powers: namely $X_{i}\subset Y_0^{\ot (2i+2)}$ and $Y_i\subset Y_0^{\ot (2i+1)}$.  Thus we may proceed inductively and define (using Frobenius reciprocity and the labeling ambiguity) for $j=1,2$ and $0\leq i\leq r-1$:
$$Y_{i+1}^{(j)}:=Y_0^{(j)}\ot X_i^{(j)},\quad X_{i}^{(j)}:=Y_0^{(j)}\ot Y_i^{(j)}.$$ 
  Thus we see that all $Y_{i}^{\paren{1}}$ and $X_{i}^{\paren{1}}$
  are tensor powers of $Y_{0}^{\paren{1}}$ and all
  $X_{i}^{\paren{2}}$ and $Y_{i}^{\paren{2}}$ are tensor powers of
  $Y_{0}^{\paren{2}}$. 
  Next, note that $Y_{i}^{\paren{2}}$ represents the isomorphism class that is the multiplicative inverse to that of $Y_{i}^{\paren{1}}$ in the Grothendieck ring of $\left[\mcC_{\mbbZ_2}\right]_0$, since $Y_j$ is self-dual.   Thus all $Y_{a}^{\paren{j}}$ for $0\leq a\leq r$ and $X_{b}^{\paren{j}}$ for $0\leq b\leq r-1$ ($j=1,2$) are in the subcategory generated by $Y_0^{(1)}$.  

  It remains to show that $F(f)=F(g)=Z$ is a tensor power of $Y_0^{(1)}$.  
  For this we compute:
  $$Y_{0}\otimes Y_{r}=Y_0\ot Y_{0}\ot f=(\one\oplus fg\oplus X_0)\ot f=X_{r-1}\oplus f\oplus g$$
  Applying the functor $F$ and the assignments above we see that $Y_0^{(j)}\ot Y_r^{(j)}=X_{r-1}^{(\overline{j})}$ for $j=1,2$ where $\overline{2}=1$ and $\overline{1}=2$.  This leaves, for $j=1,2$, $Y_0^{(j)}\ot Y_r^{(\overline{j})}=Z.$  In particular, since $Y_r^{(2)}$ is a tensor power of $Y_0^{(1)}$, we have that $Z$ is also.  Thus $\mcD_0\cong\mcC(\mbbZ_N,q)$ is a pointed cyclic modular category.
  
Now we apply the same argument as in \cite{ACRW1} and \cite{BPR} to see that the only $\mbbZ_2$ gauging on $\mcC(\mbbZ_N,q)$ that has four invertible objects is the particle-hole symmetry, i.e. inversion on $\mbbZ_N\cong \mbbZ_{2^{a}}\times\mbbZ_{p_1^{a_1}}\times\cdots\times \mbbZ_{p_b^{a_b}}$ (where $a\geq 2$.  Indeed, there can be only two fixed points under the $\mbbZ_2$ action (namely the identity and the unique element of order $2$ coming from the $\mbbZ_{2^{a}}$ factor) as otherwise we obtain $\boxtimes$ factors of the form $\mcC(\mbbZ_{p_i^{a_i}},q_i)$ in the gauged category, which is incompatible with the structure of $\mcC$.
  \end{proof}
  
\subsection{Count of inequivalent categories} The following proposition provides a count of inequivalent categories in the $4 \mid N$ case analogous to the count of  Theorem \ref{Theorem: Metaplectic main result}.
 
 \begin{prop} \label{prop:count}
 Let $N > 4$ be an integer with $4 \mid N$ and with prime factorization $N = 2^a p_1^{a_1} \cdots p_b^{a_b}$.  Then there exist precisely $3(2^{b+2})$ inequivalent metaplectic modular categories of dimension $4N$.
 \end{prop}
  \begin{proof}  
   Theorem \ref{thm:particleHoleAnalog} states that $\mcC$ is obtained from $\mcC(\mbbZ_N,q)$ for some $q$ via particle-hole gauging.  From the prime factorization of $N$, we find that there are exactly $4\cdot 2^{b}$ inequivalent categories of the form $\mcC(\mbbZ_N,q)$ (see \cite{GalindoJaramillo}).  We wish to determine the number of distinct (particle-hole) gaugings of a given fixed $\mcC(\mbbZ_N,q)$.
  
   Let $\rho: \mbbZ_2 \to \Aut(\mbbZ_N)$ denote the map determined by $\rho(1)(n) = -n$, i.e. the particle-hole symmetry.  One known extension $\mcD$ of $\mcC(\mbbZ_N, q)$ by $\rho$ has two defects $\sigma_\pm$ with fusion rules determined by
  
   \begin{itemize}
    \item $\sigma_+ \ot \sigma_+ = \bigoplus_{a \text{ even}} [a]_N$
    \item $\sigma_+ \ot \sigma_- = \bigoplus_{a \text{ odd}} [a]_N$
    \item $\sigma_\pm \ot [a] = \begin{cases}
    \sigma_\pm & [a]_N \text{ even} \\
    \sigma_\mp & [a]_N \text{ odd}
    \end{cases},$
  \end{itemize}
  where $[a]_N$ denotes the simple object corresponding to $a \in \mbbZ_N$ \cite{BBCW1}.  
  
  Given a fixed extension (in our case, $\mcD$), any extension of $\mcC(\mbbZ_N,q)$ by $\rho$ corresponds to a gauging datum, i.e. a pair $(\alpha, \beta) \in H^2_\rho(\mbbZ_2, \mbbZ_{N}) \times H^3(\mbbZ_2, U(1)) \cong \mbbZ_2 \times \mbbZ_2$ such that a certain obstruction $O_4(\rho, \alpha)$ vanishes \cite{CGPW1}.
  
  At the level of fusion rules, the action of the nontrivial element $\alpha \in H^2_\rho(\mbbZ_2, \mbbZ_{N})$ twists the tensor product of defects by a representative 2-cocycle.  One such representative cocycle is the normalized cocycle $\omega \in Z^2(\mbbZ_2, \mbbZ_{N})$ determined by $\omega(1,1) = N/2$.  The new tensor product $\otimes'$ is given by
  $$ x_g \otimes' x_h  = [\omega(g,h)]_N \otimes x_g \otimes x_h, $$
  for $x_g$ in the $g$-component of the extension and $x_h$ in the $h$-component \cite{ENO3}.  Since $\omega$ is normalized, the only non-trivial twisting occurs when both $g$ and $h$ are nontrivial.  Since $\omega(1,1) = N/2$ is even (because $4$ divides $N$),
  $$ \sigma_\pm \otimes' \sigma_\pm  = [N/2]_N \otimes \sigma_\pm \otimes \sigma_\pm = \bigoplus_{a \text{ even}} [a]_N = \sigma_\pm \otimes \sigma_\pm$$
  and 
  $$ \sigma_\pm \otimes' \sigma_\mp  = [N/2]_N \otimes \sigma_\pm \otimes \sigma_\mp = \bigoplus_{a \text{ odd}} [a]_N = \sigma_\pm \otimes \sigma_\mp.$$
  Hence, the action of $\alpha$ on the fusion rules is trivial.   However, since $\beta$ is an element of a torsor over $H^2_\rho(\mbbZ_2,\mbbZ_N)$, the choice of $\alpha$ is still relevant. 
  
  There are two potential obstructions to extending $\mcC(\mbbZ_N,q)$ by $\rho$.
  The first obstruction in $H_\rho^3(\mbbZ_2,\mbbZ_N)$ vanishes since we know that a gauging exists.  The second obstruction vanishes since $H^4(\mbbZ_2,U(1)) \cong 0$.  There is only one possible set of fusion rules on the $\mbbZ_2$-extension since the action of $H_\rho^2(\mbbZ_2,\mbbZ_N)$ on the fusion rules is trivial. There is a choice of associativity constraints on the $\mbbZ_2$-extension, so that \emph{a priori} we have $4$ distinct theories.  However, by examining the pairs of FS-indicators of the defects $\sigma_\pm$ as in \cite[Section X.F]{BBCW1} we find that they are $(1,1),(-1,-1),(1,-1)$ and $(-1,1)$, so that the latter two theories can be identified by relabeling.  Thus we have $3$ distinct gaugings for each fixed $\mcC(\mbbZ_N,q)$, yielding a count of $3(2^{b+2})$ metaplectic modular categories of dimension $4N>16$ with $4\mid N$.  
\end{proof}

\subsection{Degenerate case: metaplectic modular categories of dimension $16$}

As we have observed above, metaplectic categories of dimension $16$ are of the form $\Ising^{(\nu_1)}\boxtimes \Ising^{(\nu_2)}$ where $\nu_i\in\mbbZ_{16}^{\times}$.  In that case we have $Y_0^{\ot 2}=\one\oplus f\oplus g\oplus fg$, so that we cannot determine $Y_0^{(1)}\otimes Y_0^{(2)}$ from $(Y_0^{(1)})^{\ot 2}$ using only the fusion rules--one must be $\one_\mcD$ and the other $Z$.  The two possiblities lead to $\mcD_0\cong\mcC(\mbbZ_4,q_1)$ or $\mcD_0\cong\mcC(\mbbZ_2\times\mbbZ_2,q_2)$.  Both of these can occur: see \cite{BBCW1}.
 
The following proposition provides a detailed count for the $N = 4$ case. 
\begin{prop}
  \label{prop:Ising2}
  There are precisely 20 inequivalent modular categories of the form $\Ising^{\nu_1}\boxtimes\Ising^{\nu_2}$: 
  \begin{itemize}
  \item 12 come from gauging the particle-hole symmetry of a modular category of the form $\mcC(\mbbZ_4,q)$
  \item 8 come from gauging a modular category of the form $\mcC(\mbbZ_2 \times \mbbZ_2, q)$ containing a fermion, 2 from each of the following categories:  
  \begin{itemize}
  \item TC (Toric Code), 
  \item 3F (3 fermions), 
  \item $\Sem^2$ (semion-squared),
  \item and $\overline{\Sem}^2$ (semion-conjugate-squared).
  \end{itemize}
  \end{itemize}
\end{prop}
\begin{proof}
 Notice that when $N=4$ the argument of Proposition \ref{prop:count} is still valid if applied only to $\mcC(\mbbZ_4,q)$: we still obtain $12$ theories of the form $\Ising^{\nu_1}\boxtimes\Ising^{\nu_2}$ from these.  By counting gaugings as in \cite{BBCW1}, we also get 8 more $\Ising^{\nu_1}\boxtimes\Ising^{\nu_2}$ theories from gauging: two from each of the $4$ distinct $\mcC(\mbbZ_2\times \mbbZ_2,q)$ modular categories containing a fermion.  Thus we obtain a count of $20$ metaplectic modular categories of dimension $16$.
\end{proof}

Alternatively, we can count $T$-matrices up to dimension-preserving permutation to get a lower bound on the count of inequivalent categories.  As shown below, this lower bound coincides with the gauging-based count, resulting in the following proposition.

\begin{prop}
A category of the form $\Ising^{\nu_1} \boxtimes \Ising^{\nu_2}$ is determined by its modular data.
\end{prop}
\begin{proof}
 Any equivalence of modular categories $\widehat F :\Ising^{\nu_1}\boxtimes\Ising^{\nu_2} \to \Ising^{\nu_1'}\boxtimes\Ising^{\nu_2'}$  induces a dimension- and twist-preserving map  $F:\Irr(\Ising^{\nu_1}\boxtimes\Ising^{\nu_2}) \to \Irr(\Ising^{\nu_1'}\boxtimes\Ising^{\nu_2'})$ on the sets of isomorphism classes of simple objects.  Thus, by identifying corresponding simple objects in the domain and codomain of $F$, the $T$-matrices of equivalent modular categories are related by dimension-preserving permutations of their basis, the set of isomorphism classes of simple objects. Thus, a lower bound for a count of inequivalent categories is given by counting $T$-matrices up to dimension-preserving permutation. 

To make this count, we use the following table of dimensions and twists of simple objects of $\Ising^{\nu_1} \boxtimes \Ising^{\nu_2}$:

$$
    \begin{tabular}{>{$}l<{$}|*{3}{>{$}l<{$}}}
    (\dim_{X \boxtimes Y}, \theta_{X \boxtimes Y})  & \one  & \psi   & \sigma  \\
    \hline\vrule height 12pt width 0pt
    \one       & (1,1)            &  (1, -1)       & (\sqrt{2}, e^\frac{\pi i \nu_2}{8}) \\
    \psi    & (1, -1)          &  (1, 1)         & (\sqrt{2}, -e^\frac{\pi i \nu_2}{8})\\
    \sigma  & (\sqrt{2}, e^\frac{\pi i \nu_1}{8})   & (\sqrt{2}, -e^\frac{\pi i \nu_1}{8})  & (2, e^\frac{\pi i (\nu_1 + \nu_2)}{8})  \\
    \end{tabular} 
$$

We claim that the map $F:\Irr(\Ising^{\nu_1}\boxtimes\Ising^{\nu_2}) \to \Irr(\Ising^{\nu_1'}\boxtimes\Ising^{\nu_2'})$ is determined by $F(\sigma \boxtimes \one)$. Since $F$ must preserve simple objects of dimension $\sqrt{2}$, we have $F(\sigma \boxtimes \one)  \in \{\sigma \boxtimes \one, \one \boxtimes \sigma, \psi \boxtimes \sigma, \sigma \boxtimes \psi\}$.   Thus, $F((\psi \oplus \one) \boxtimes \one) =  F(\sigma^{\otimes 2} \boxtimes \one)  =  (F(\sigma \boxtimes \one))^{\otimes 2} \in \{(\psi \oplus \one) \boxtimes \one, \one \boxtimes (\psi \oplus \one) \}$. Hence,  $F(\sigma \boxtimes \one)$ determines $F(\psi \boxtimes \one) \in \{\psi \boxtimes \one, \one \boxtimes \psi\}$. The value of $F(\psi \boxtimes \one)$ also determines $F(\one \boxtimes \psi)$ since these are the only fermions.  Hence, $F(\sigma \boxtimes \one)$ also determines $F(\sigma \boxtimes \psi)$.  Since $F$ must preserve the unique simple object of dimension $2$, we have $\nu_1' + \nu_2' = \nu_1 + \nu_2 \pmod{16}$.  Since the twist of $F(\sigma \boxtimes \one)$ is $\nu_1 \in \{\nu_i', \nu_i' + 8\}$ for some $i \in \{1,2\}$, this equation determines $\nu_2 \in \{\nu_{3-i}', \nu_{3-i}' + 8\}$, hence the value of $F(\one \boxtimes \sigma)$.  The other simple objects are generated by these values, proving the claim.

We thus have the following count of $T$-matrices up to dimension-preserving permutation.  There are $8 \cdot 8 = 64$ ordered pairs $(\nu_1, \nu_2)$.  
This splits into $56$ pairs of distinct $(\nu_1, \nu_2)$ and $8$ pairs of the
form $(\nu, \nu)$. Since $\nu$ is odd, the orbit of $(\nu,\nu)$ under dimension- and twist-preserving permutation is $\{(\nu,\nu), (\nu + 8, \nu + 8)\}$ of order $2$.   The $56$ pairs of distinct $(\nu_1, \nu_2)$ further break up into $8$ pairs such that $\nu_1 + 8 = \nu_2$ and $48$ pairs such that  $\nu_1 + 8 \neq \nu_2$.  In the case that $\nu_1 + 8 = \nu_2$, the orbit of $(\nu_1, \nu_2)$ is $\{(\nu_1, \nu_2), (\nu_2, \nu_1)\}$ of order $2$. In the case that $\nu_1 + 8 \neq \nu_2$, the orbit is $\{(\nu_1, \nu_2), (\nu_2, \nu_1), (\nu_1 + 8, \nu_2 + 8), (\nu_2 + 8, \nu_1 + 8)\}$ of order $4$.  Thus, we have
$8/2 + 8/2 + 48/4 = 20$ distinct $T$-matrices up to dimension-preserving permutation.  In light of Proposition \ref{prop:Ising2}, this implies that the choice of $T$-matrix determines the category.
\end{proof}

More generally, we would like to know if the modular data distinguishes inequivalent metaplectic modular categories of arbitrary dimension.  A first step in this direction would be to determine the modular data:
\begin{question}
  \label{Question: DefectTwists}
  What are the twists of each metaplectic modular category of dimension $4N$ in terms of the twists of the pointed modular category $\mcC(\mbbZ_N,q)$?
\end{question}

  Formula (412) of \cite{CGPW1} conjecturally relates twists in the
  gauged theory to twists in the extension.  This still leaves the problem
  of finding the twists of the defects in the extension.  As far as we can
  tell, the only way to do this is to solve the $G$-crossed heptagon
  equations.

\section{Modular categories of dimension $16m$, with $m$ odd square-free integer}
\label{16m}

In this section we will consider $\mathcal{C}$ a strictly weakly-integral modular category of dimension $16m$, with $m>1$ an odd square-free integer. We will fix this assumptions for all Section \ref{16m} unless otherwise specified. We will make a comment on the case $m=1$, see Remark \ref{16m_integral} and on the integral case in the discussion below. 

We have seen examples, namely metaplectic modular categories of dimension $4N$ with $N=4m$ with $m$ square-free and odd. The goal is to give a classification of this class of categories similar to the ones in \cite{BGNPRW1,BPR} of modular categories of dimension $4m$ and $8m$, respectively. In what follows, we give a classification of modular categories of dimension $16m$ under certain restrictions.

The only \emph{integral} metaplectic modular categories of dimension $16m$ with $m$ square-free are those with $m=1$.  We will disregard this case, pausing only to ask:
\begin{question}
  \label{Question: 16m integral}
 Are integral modular categories of dimension $16m$ with $m > 1$ and not necessary square-free, pointed?
\end{question}
\begin{rmk}\label{16m_integral}
The only integral modular categories of dimension $16m$ are for $m=1$, i.e. dimension $16$. But there are strictly weakly-integral modular categories of dimension $16$ as well, for example $\Ising\boxtimes \Ising$.
\end{rmk}
\begin{rmk} While \textit{integral} modular categories of dimension $16$ are pointed (\cite[Lemma 4.3]{BPR}), it is not the case that all integral modular categories of  dimension $2^{k}$ are pointed. For example $SO(8)_2$ is a non-pointed integral modular category of dimension $32$, which has $4$ invertible objects and $7$ simple objects of dimension $2$. 
\end{rmk} 


We can consider $\mathcal C$ \textit{prime} (i.e. not of the form $\mcC_1\boxtimes\mcC_2$) since otherwise it reduces to known cases, see \cite{BGNPRW1,BPR}.  

Prime modular categories of dimension $4m$ and $8m$ essentially arise from metaplectic categories when they are not pointed. Thus it is natural to ask:

\begin{question}
  \label{Question: 16m}
 Are prime modular categories of dimension $16m$ Grothendieck equivalent to $\SO\paren{4m}_{2}$?
 \end{question}
Strictly weakly integral modular categories of dimension $16$ were classified in \cite[Lemma 4.9]{BPR}. These categories are a Deligne product of an Ising category and a pointed modular category, i.e. a Generalized Tamabra-Yamagami category. 

\subsection{The order of $U(\mathcal{C})$}

From \cite[Lemma 4.1]{BPR}, we know that if $16$ divides the order of the universal grading group $U(\mathcal C)$ of $\mathcal{C}$, then $\mathcal{C}$ is pointed. So we can assume that $16 \nmid  \left|U(\mathcal{C}) \right|$.

As in the proof of Theorem 3.1 in \cite{BGNPRW1}, using the de-equivariantization process, we can assume that $\left| U(\mathcal{C}) \right| = 2^k$, for $k$ an integer number.

We will not provide a complete answer to Question \ref{Question: 16m}, but we will give a partial response and pose some general questions arising from our analysis.

\begin{remark}\label{dim in Cint} Let $\mcC$ have dimension $2^nm$, with $n\geq 3$. Then, the possible dimensions of integral simple objects of $\mcC$ are $2^k$, with $k = 0, \ldots, \lfloor \frac{n}{2} \rfloor$ (see e.g. \cite[Lemma 5.2]{DN}, \cite[Theorem 2.11 (i)]{ENO2}).
\end{remark}

\begin{lemma}\label{Lemma: universal grading lower bound}
  Let $\mcC$ have dimension $2^nm$, with $n\geq 3$. Then the universal grading group $\mcU(\mcC)$ has
   $|\mcU(\mcC)|\geq 4$.
\end{lemma}
\begin{proof} Since $\mcU(\mcC)$ is non-trivial and the GN grading group is also non-trivial, it is enough to show that $\mcU(\mcC)\not\cong\mbbZ_2$.
Suppose that $\mcC$ is as in the statement with $\mcU(\mcC)\cong\mbbZ_2$. Notice that the GN-grading group is also $\mathbb Z_2$. By faithfulness of the universal
  grading we have $\dim \mcC_{ad}= 2^{n-1}m$ (where $n\geq 3$). 
   From the definition of dimension and Remark \ref{dim in Cint}, we have $2^{n-1}m=2+\sum_{k = 1}^{\lfloor \frac{n}{2} \rfloor} 2^{2k}a_k$, where $a_k$ is the number of $2^k$-dimensional simple objects in $\mcC_{ad}$. Since $n>3$ and $4\nmid 2$ we have a contradiction.
\end{proof}

\begin{lemma}\label{Lemma: universal grading upper bound}
Let $n\geq 3$. There are no prime self-dual strictly weakly integral modular categories of dimension $2^nm$ whose universal
grading group has order $2^{n-1}$.
\end{lemma}
\begin{proof}
By contradiction, assume that such a category exists.  Since $\mcC$ is self-dual, we must have that  $\mcU\paren{\mcC}\cong\mbbZ_{2}^{n-1}$. So let $\mcC_{g}$ be a non-integral component. Then $\mcC_{g}\otimes \mcC_{g}\subset\mcC_{ad}$. Thus
  $\mcD:=\mcC_{ad}\oplus\mcC_{g}$ is a fusion category of dimension $4m$. Taking
  centralizers we have
  $\mcZ_{\mcD}\paren{\mcD}\subset\mcZ_{\mcD}\paren{\mcC_{ad}}=\paren{\mcC_{ad}}_{pt}$.

  Since $\mcC_{ad}$ has dimension $2m$, we have $2m=\sum_{k = 0}^{\lfloor \frac{n}{2} \rfloor} 2^{2k}a_k$, where $a_k$ is the number of $2^k$-dimensional simple objects in $\mcC_{ad}$, by Remark \ref{dim in Cint}.
  Thus $a_{0}=2$ because $m$ is odd.
  So either $\mcZ_{\mcD}\paren{\mcC_{ad}\oplus\mcC_{g}}\cong \Vec$, $\sVec$, or
  $\Rep\paren{\mbbZ_{2}}$.

  The case $\mcZ_{\mcD}\paren{\mcD}\cong \Vec$ is not possible as $\mcC$ is prime.  Since there are only $2$ invertible objects in $\mcC_{ad}$ and all non-invertibles objecs in $\mcC_{ad}$ are even dimensional, then they should be fixed by the non-trivial invertible object in $\mcC_{ad}$, by \ref{formula Xotimesdual}. Since the invertible objects are transparent in $\mcC_{ad}$, the non-trivial invertible object can not be a fermion, i.e. $\mcZ_{\mcD}\paren{\mcD}\ncong \sVec$. If $\mcZ_{\mcD}\paren{\mcD}\cong \Rep\paren{\mbbZ_{2}}$, we can
  de-equivariantize to find a $2m$-dimensional modular category. Thus
  $\mcD_{\mbbZ_{2}}$ is pointed (by \cite[Theorem 2.11 (i)]{ENO2}) and hence $\mcD$ is integral, an impossibility.
\end{proof}

\begin{prop}\label{universal grading group order 4}
 If $\mcC$ is prime and of dimension $16m$, then its universal grading group has order $4$. Thus, $\FPdim \mcC_{pt}=4$.
\end{prop}
\begin{proof}
This statement is a direct consequence of Lemma \ref{Lemma: universal grading lower bound} and Lemma \ref{Lemma: universal grading upper bound}.
\end{proof}

\subsection{Self-dual and GN-grading by $\mathbb Z_2$ case}

We will focus on the self-dual case. Thus, it suffices to consider strictly weakly integral self-dual categories of dimension $16m$ whose universal grading group has order $4$.

\begin{lemma}
Let  $\mcC$ have dimension $2^nm$, with $n\geq 4$. If $\dim\mcC_{pt}=4$, then
$\mcZ_{\mcC_{ad}}\paren{\mcC_{ad}}=\mcC_{pt}$, and $\mcC_{pt}$ contains a
boson.
\end{lemma}
\begin{proof}
  First note that $\dim\mcC_{ad}=2^{n-2}m$. So the dimensions of the simple objects of $\mcC_{ad}$ are of the form $2^k$, $k= 0, \ldots, \lfloor \frac{n}{2} \rfloor$, by Remark \ref{dim in Cint}. Thus $2^{n-2}m=\sum_{k = 0}^{\lfloor \frac{n}{2} \rfloor} 2^{2k}a_k$. Since $n\geq 4$, then $4\mid a_0$. Moreover $a_{0}=4$ because
  $\one\in\mcC_{ad}$. Thus $\mcC_{pt}\subset\mcC_{ad}$. On
  the other hand,
  $\mcZ_{\mcC_{ad}}\paren{\mcC_{ad}}=\mcZ_{\mcC}\paren{\mcC_{ad}}\cap\mcC_{ad}=\mcC_{pt}\cap\mcC_{ad}=\mcC_{pt}$.

  Now consider $X$ a non-invertible simple object in $\mcC_{ad}$ (which exists since $m>1$). 
  Recall that $2\mid \FPdim Y$, for all non-invertible simple object $Y\in \mcC_{ad}$. Then, by Remark \ref{formula Xotimesdual}, we have that $2$ divides the order of the group $G[X]$ of invertible objects that fixes $X$. Since $\one \in G[X]$, then $G[X] \geq 2$
  Then, there exists $g\in G[X] \subseteq \mcC_{pt} = \mcZ_{\mcC_{ad}}\paren{\mcC_{ad}}$ of order $2$. Since $g$ fixes $X$ and $g\in\mcC_{ad}^\prime=\mcC_{pt}$ then $g$ is not a fermion (see \cite[Lemma 5.4]{M4}), and hence $g$ must be a
  boson.
\end{proof}

\begin{lemma}
  \label{Lemma: All 2-dimensionals}
  If $\mcC$ has dimension $16m$ with $\dim\mcC_{pt}=4$ and
  GN-grading $\mbbZ_{2}$, then the simple objects in the integral non-adjoint component are all $2$-dimensional. Moreover, all the non-invertible simple objects in $\mcC_{ad}$ are also $2$-dimensional.
\end{lemma}
\begin{proof}
  Let $a_{0}$, $b_{0}$, and $c_{0}$ denote the number of 1-dimensional,
  2-dimensional, and 4-dimensional simples in $\mcC_{ad}$ respectively. Then
  $4m=\dim\mcC_{ad}=a_{0}+4b_{0}+16c_{0}$. In particular, $4\mid a_{0}$. Since $\one$ is in $\mcC_{ad}$ then $a_{0}>0$. Thus $a_{0}=4$. Consequently,
  $\mcC_{g}$, the integral non-adjoint component must contain only
  $2$-dimensional and $4$-dimensional simple objects.

  Moreover, $4m=4b_{1}+16c_{1}$ where $b_{1}$ and $c_{1}$ are the number of 2-
  and 4-dimensionals in the integral non-adjoint component. Since $m$ is odd we
  know that $b_{1}\neq 0$, but $m=b_{1}+4c_{1}$, so $b_{1}$ is odd.

  Since the GN-grading group and the universal grading group are not the same then $\mcC_{ad}\subsetneq\mcC_{int}$. So we have
  $\mcZ_{\mcC}\paren{\mcC_{int}}\subsetneq\mcZ_{\mcC}\paren{\mcC_{ad}}=\mcC_{pt}$.
   Notice that $\mcC_{int}$ is not modular, otherwise, by \cite[Theorem 4.2]{M1}, \cite[Theorem 3.13]{DGNO1}, the category $\mcC \cong \mcC_{int} \boxtimes \langle g\rangle$ would be integral which contradicts the assumption of $\mcC$ being strictly weakly integral. Thus $\mcZ_{\mcC}\paren{\mcC_{int}}=\langle g\rangle$ for some invertible $g$.

  We saw above that there is an odd number of 2-dimensional simple objects in the
  non-adjoint integral component. Then, at least one these $2$-dimensional objects are fixed by $g$.
  It follows from \cite[Lemma 5.4]{M4} that $g$ is a boson. Thus
  $\mcC_{int}$ is modularizable.

  So $\paren{\mcC_{int}}_{\mbbZ_{2}}$ is a $4m$-dimensional integral modular category and
  thus is pointed, by \cite[Theorem 3.1]{BGHKNNPR1}. In particular, $\mcC_{int}$
  has character degrees $1$ and $2$.
\end{proof}


\begin{prop}
  \label{Prop: Must be Spin}
  If $\mcC$ is a self-dual prime modular category of dimension $16m$ with 
  GN-grading $\mbbZ_{2}$, then $\mcC$ is spin modular.
\end{prop}
\begin{proof}
  Suppose that the non-trivial invertibles, $g$, $h$, and $gh$, are all bosons.
  
  In the same way as in the proof of Lemma \ref{Lemma: All 2-dimensionals} it can be shown that $\mcZ_{\mcC}\paren{\mcC_{int}}=\langle g\rangle$, for some invertible $g$.
 Since the GN-grading group is $\mathbb Z_2$, the possible non-integral dimensions of simple objects are $\sqrt{t}$ and $2\sqrt{t}$, for some $t\in \mathbb N$. 
 
  
  Denote by $\mcD_{a}$, the
  subcategory $\mcC_{ad}\oplus\mcC_{a}$. Then
  $\mcZ_{\mcC}\paren{\mcD}\subset\mcZ_{\mcC}\paren{\mcC_{ad}}=\mcC_{pt}$. By
  primality and taking double centralizers we can conclude that
  $\mcZ_{\mcC}\paren{\mcD}\neq \Vec$, $\mcC_{pt}$. In particular, there is a
  boson $b$ such that $\mcZ_{\mcC}\paren{\mcD}=\langle b\rangle$. Once again, by
  double centralizing, $b\neq g$. 
  Similarly, denoting $\mcD_{c}$, the
  subcategory $\mcC_{ad}\oplus\mcC_{c}$, where $\mcC_{c}$ is the other non-integral component, we get $\mcZ_{\mcC}\paren{\mcD}=\langle bg \rangle$. With out lost of generality, we can assume that $b = h$.
  
  Since $g$ does not centralize $X$ we have
  $S_{g,X}\neq d_{X}$ and so by balancing and orthogonality of the S-matrix, $g$ must move $X$. Similarly, $g$ must
  move $Y$. On the other hand, $g$ fixes the 2-dimensional simple objects in $\mcC_{ad}$. 

  Now, assume there exist a simple $X\in \mcD_a$ of dimension $\sqrt{t}$ and a simple $Y\in \mcD_c$ of dimension $2\sqrt{t}$.
  Then $X\otimes Y$ is in the non-adjoint integral component of $\mcC$. So $X\otimes Y$ is a sum of $2$-dimensional simple objects. But if $Z$ is a $2$-dimensional simple object such that $Z\otimes X \cong Y$ then $Y \cong Z\otimes Y\cong g\otimes Z\otimes X \cong g\otimes Y$. This is a contradiction since $g$ does not fix $Y$.
  
Then, there is a fermion on $\mcC$. Moreover, both $h$ and $gh$ are fermions in $\mcC$.
\end{proof}

\begin{cor}
  \label{Corollary: subcat centralizers}
If $\mcC$ is a self-dual prime modular category of dimension $16m$ with GN-grading $\mbbZ_{2}$. Let $h$ and $gh$ be the fermions in $\mcC$. Denote by $\mcC_{a}$ and $\mcC_{b}$ the
  non-integral components, then
  $\mcZ_{\mcC}\paren{\mcC_{ad}\oplus\mcC_{a}}=\langle h\rangle$ and
  $\mcZ_{\mcC}\paren{\mcC_{ad}\oplus\mcC_{b}}=\langle gh\rangle$. 
  Moreover, there is no component containing only objects of dimension $2\sqrt{t}$ for some square-free integer $t$.
\end{cor}
\begin{proof}

  The proof of the first statement is contained in the proof of Proposition \ref{Prop: Must be Spin}.
  
  By \cite[Theorem 3.10]{GN2}, the only possible dimensions of non-integral
  objects are $\sqrt{t}$ and $2\sqrt{t}$ for a square-free integer $t$. 
  
  Suppose to the contrary that there is some component containing only simples of
  dimension $2\sqrt{t}$. Without loss of generality we take this to be the $\mcC_{a}$ component. Then tensoring with $h$ must permute these simples in a
  fixed point free manner. Then $4m = \FPdim \mcC_a = 4tk$, where $k$ is the number of simples in $\mcC$. This is a contradiction since $k$ is even and $m$ odd.
\end{proof}
\subsection{Classification for dimension $16m$}
\begin{lemma}
Let $\mcC$ be a self-dual prime modular category of dimension $16m$ with GN-grading $\mbbZ_{2}$.
Then the non-integral objects have dimension $\sqrt{t}$
and $2\sqrt{t}$ for some \emph{square-free even} integer $t$. Moreover each non-integral component contains an even number of objects of dimension $\sqrt{t}$ and an even number of objects of dimension $2\sqrt{t}$.

\end{lemma}
\begin{proof}
By \corref{Corollary: subcat centralizers} we know that there is at least one
  object of dimension $\sqrt{t}$ in each non-integral component, let $X$ and $Y$
  be such objects in $\mcC_{a}$ and $\mcC_{b}$ (under the notation of
  \corref{Corollary: subcat centralizers}), respectively. Then $X\otimes Y$ must be in the
  integral non-adjoint component under the universal grading. In particular, by Lemma \ref{Lemma: All 2-dimensionals}, $t=\dim X\otimes Y=2s$ for some $s\in\mbbZ$. The $t$ is even. Furthermore, $t$ is square-free since $t$ divides $m$.

 By \propref{Prop: Must be Spin}, $\mcC$ is a spin modular category. Denote $\mcD=\mcC_{ad}\oplus\mcC_{a}$, then
  $\mcZ_{\mcD}\paren{\mcD}\cong\sVec$ and even multiplicity follows from
  the fact that the action of the fermion is fixed-point-free. The same holds for $\mcC_{ad}\oplus\mcC_{b}$.
%
%
%
\end{proof}

\begin{cor}
Let $\mcC$ be a self-dual prime modular category of dimension $16m$ with GN-grading $\mbbZ_{2}$.
Then all of the non-integral simple objects have dimension $\sqrt{2m}$.
\end{cor}
\begin{proof}
By \propref{Prop: Must be Spin}, $\mcC$ must be spin modular.

Suppose there are objects $X$ and $Y$ in the same component of dimension $2\sqrt{t}$ and $\sqrt{t}$ respectively. Then $X\otimes Y$ is in
  $\mcC_{ad}$. Since $X$ and $Y$ have different dimensions, no invertibles can
  appear as subobjects of $X\otimes Y$. Thus $X\otimes Y=\bigoplus
  N_{X,Y}^{Z_{i}}Z_{i}$ where $Z_{i}$ are the 2-dimensional simple objects in
  $\mcC_{ad}$. But $N_{X,Y}^{Z_{i}} = N_{X,Z_{i}}^{Y}$.
  If we consider $Z_i\otimes X$ there are $3$ possibilities: it is the direct sum of $2$ simple objects of dimension $\sqrt{t}$, it is equal to a simple object of dimension $2\sqrt{t}$, or it equal to $2$ copies of a simple object of dimension $\sqrt{t}$.   If we tensor $Z\otimes X$ by the boson $g$, since $Z$ is fixed by $g$, we get that $Z\otimes X = X_1\oplus gX_1$, where $X_1$ is a $\sqrt{t}$-dimensional simple object in $\mcC_a$. This implies that $N_{X,Y}^{Z_{i}} = N_{X,Z_{i}}^{Y}= 0$ for all $X$, $Y$, and $Z$ as above. But this is a contradiction because $X\otimes Y$ has dimension $2t$.
  It follows that the non-integral objects all have
  dimension $\sqrt{t}$ for some even square-free integer $t$, see \corref{Corollary: subcat centralizers}.

  Now suppose there exist non-integral simple objects $W\neq V$ from the same
  component such that $W\otimes V$ does not contain any invertible objects. Then
  the fermion that doesn't centralize $W$ and $V$ must fix them. On the other
  hand, this fermion must permute all elements of the adjoint subcategory since it is transparent in $\mcC_{ad}$. Hence $W\otimes V$ must decompose into an even number of 2-dimensional simples
  in the adjoint. Computing dimensions we see $4\mid \dim W\otimes V=t$, which is impossible since $t$ is square-free. Thus for any non-integral $W,V$ in the same component,
  $W\otimes V$ contains an invertible, say $a$. Thus $a\otimes W=V$. In
  particular, this component can only contain two non-integral objects. The
  result now follows by equidimensionality of the universal grading.
\end{proof}

We see that under the assumptions above, the category $\mcC$ has the following structure, where we take $\mbbZ_2\times\mbbZ_2$ grading: 
\begin{itemize}
\item $\mcC_0$ has 1 boson, 2 fermions and $m-1$ objects of dimension $2$.
\item $\mcC_{(1,1)}$ has $m$ objects of dimension $2$
\item $\mcC_{(0,1)}$ and $\mcC_{(1,0)}$ each have 2 objects of dimension $\sqrt{2m}$.  
\end{itemize}
The 
equation together with Corollary \ref{Corollary: subcat centralizers} shows that the twists of the simple objects in $\mcC_{(1,0)}$ and $\mcC_{(0,1)}$ come in pairs of the form $(\theta,-\theta)$.  Similarly, the $m-1$ simple objects of dimension $2$ in $\mcC_{0}$ come in pairs with twists of the form $(\theta,-\theta)$.  In particular the central charge of $\mcC$ only depends on the twists in $\mcC_{(1,1)}$.  This is not quite sufficient to conclude that $\mcC$ is metaplectic, but for certain the trivial component of the $\mbbZ_2$-de-equivariantization has $[\mcC_{\mbbZ_2}]_0\cong \mcC(A,q)$ where $A$ is an abelian group of order $4m$, and the corresponding action of $\mbbZ_2$ by braided tensor autoequivalences on $\mcC(A,q)$ has exactly 2 fixed points.



Although we do not get as sharp a result in the case of $16m$ as we do in the cases $4m$ and $8m$, we can still give a significant amount of the structure under some assumptions:
\begin{theorem} Let $\mcC$ be a prime modular category of dimension $2^nm$ with $m$ odd and square free. \begin{itemize}
    \item If $n=1$, then $\mcC$ is pointed.
    \item \cite{BGNPRW1,BPR} If $n=2$ or $3$, then $\mcC$ is a metaplectic modular category.
    \item If $n=4$ and $\mcC$ is self-dual and its $GN$ grading is $\mbbZ_2$, then 
    \begin{itemize} \item $\mcC$ is  metaplectic, or 
    \item $\mcC$ is obtained as a $\mbbZ_2$ gauging of $\mcC(\mbbZ_2\times\mbbZ_{2m},q)$ via an action of $\mbbZ_2$ with exactly $2$ fixed points. 
    \end{itemize}
    \end{itemize}
\end{theorem}
We do not know if the $\mbbZ_2$ gauging of $\mcC(\mbbZ_2\times\mbbZ_{2m},q)$ is self-dual or prime, so it is possible that a sharper result can be obtained.  Moreover, we do not know what happens if the GN-grading is $\mbbZ_2\times\mbbZ_2$.

\subsection{Related questions}
\begin{question}
  \label{Question: Big GN}
  How large can the GN-grading group be for a prime strictly weakly integral
  modular category?
\end{question}

There is an upper bound, given in \cite[Theorem 3.1]{DN}, for weakly-integral modular categories of dimension $2^nm$, where $n \geq 0$ and $m$ odd. In this case, the order of the GN-grading group is at most $2^{\frac{n}{2}}$.

It is of course straight forward to develop non-prime categories with large
GN-grading via Deligne product of $\mbbZ_{2}$-graded categories. A natural place
to look for large GN-grading with fewer prime factors is through
equivariantization. That is, perhaps one can find an equivariantization of one
of these Deligne products that has a diverse prime factor. This of course leads
to the following related question:
\begin{question}
  \label{Question: Factorizations under (de-)equivariantization}
  How do Deligne products behave under (de-)equivariantization?
\end{question}

\bibliographystyle{abbrv}

\end{document}